\documentclass[12pt]{amsart}
\usepackage{amssymb,verbatim,enumerate,ifthen}
\usepackage{color}
\usepackage[mathscr]{eucal}
\usepackage[utf8]{inputenc}
\usepackage[T1]{fontenc}
\usepackage{esint}
\usepackage{marginnote}
\usepackage{dsfont}
\newtheorem{thm}{Theorem}[section]
\newtheorem{cor}[thm]{Corollary}
\newtheorem{prop}[thm]{Proposition}
\newtheorem{lem}[thm]{Lemma}

\theoremstyle{definition}

\newtheorem{rem}[thm]{Remark}
\newtheorem{exa}[thm]{Example}

\numberwithin{equation}{section}
\def\eq#1{{\rm(\ref{#1})}}
\def\Eq#1#2{\ifthenelse{\equal{#1}{*}}
  {\begin{equation*}\begin{aligned}[]#2\end{aligned}\end{equation*}}
  {\begin{equation}\begin{aligned}[]\label{#1}#2\end{aligned}\end{equation}}}

  \newcommand{\upRiemannint}[2]{
  \overline{\int_{#1}^{#2}}
}

\def\A{\mathscr{A}}
\def\iM{\mathcal{M}}

\def\D{\mathscr{D}}

\def\G{\mathscr{G}}
\def\M{\mathscr{M}}
\def\Nm{\mathscr{N}}
\def\LM{{\M^L}}
\def\UM{{\M^U}}
\def\P{\mathscr{P}}
\newcommand{\operator}[1]{\mathop{\vphantom{\sum}\mathchoice
{\vcenter{\hbox{\LARGE $#1$}}}
{\vcenter{\hbox{\Large $#1$}}}{#1}{#1}}\displaylimits}

\DeclareMathOperator\esssup{esssup}
\DeclareMathOperator\essinf{essinf}
\def\Mm{\operator{\mathscr{M}}}
\def\Ar{\operator{\mathscr{A}}}       
\newcommand\R{\mathbb{R}}
\newcommand\N{\mathbb{N}}
\newcommand\Z{\mathbb{Z}}
\newcommand\Q{\mathbb{Q}}
\newcommand\K{\mathbb{K}}
\newcommand\calL{\mathcal{L}}
\thanks{I gratefully thank the anonymous referee for his/her detailed comments which greatly helped us to improve the presentation.}
\newcommand{\norm}[1]{\left\| #1 \right\| }
\newcommand{\abs}[1]{\left| #1 \right| }

\newcommand{\tauto}{\overset{\tau}\to}

\DeclareMathOperator{\cl}{cl}
\DeclareMathOperator{\Rg}{Rg}
\DeclareMathOperator{\Dom}{Dom}
\DeclareMathOperator{\Graph}{Graph}
\DeclareMathOperator{\RG}{RG}

\DeclareMathOperator{\interior}{int}

\newcommand{\iHc}{\mathbf{H}}

\newcommand{\Hc}{\mathscr{H}}
\newcommand{\Est}[2][SKIPPED]{
\ifthenelse{\equal{#1}{SKIPPED}}
  {
    \ifthenelse{\equal{#2}{}}
      {\mathscr{C}}
      {\mathscr{C}(#2)}
  }
  {
    \ifthenelse{\equal{#2}{}}
      {\mathscr{C}_{#1}}
      {\mathscr{C}_{#1}(#2)}
  }
}
\author[P. Pasteczka]{Pawe\l{} Pasteczka}
\address{Institute of Mathematics \\ Pedagogical University of Krakow \\ Podchor\k{a}\.zych str. 2, 30-084 Krak\'ow, Poland}
\email{pawel.pasteczka@up.krakow.pl}

    

\title{On the integral approach to means and their Hardy property}
\begin{document}
\begin{abstract}
The celebrated Hardy inequality can be written in the form 
$$\int_0^\infty \mathcal{P}_p \big(f|_{[0,x]}\big)dx \le (1-p)^{-1/p} \int_0^\infty f(x)\:dx
\qquad \text{ for }p\in(0,1)\text{ and }f \in L^1\text{ with }f\ge0,$$
where $\mathcal{P}_p$ stands for the $p$-th power mean. One can ask about possible generalizations of this property to another families (with sharp constant depending on the mean).

Adapting the notion of Riemann integral, for every weighted mean we define the lower and the upper integral mean. We prove that every symmetric, monotone, $\mathbb{R}$-weighted mean on $I$ which is continuous in its entries and weights has at most one continuous extension to the integral one. Moreover this extension preserves the Hardy constant. 
This result allows to extend the latter inequality to the family of homogeneous, concave deviation means.

\end{abstract}
\maketitle


\section{Introduction}
Classical Hardy inequality \cite{Har20} is expressed in two forms (see also Landau \cite{Lan21})
\Eq{*}{
\sum_{n=1}^\infty \Big(\frac{a_1+\dots+a_n}n\Big)^p &\le \Big(\frac{p}{p-1}\Big)^{\frac1p} \sum_{n=1}^\infty a_n^p
&\qquad \text{ for }p>1\text{ and }a \in \ell^p\text{ with }a_n\ge0;\\
\int_0^\infty \Big(\frac1x \int_0^x f(t)\:dt\Big)^{p}\:dx &\le \Big(\frac{p}{p-1}\Big)^{\frac1p} \int_0^\infty f(x)^p \:dx
&\qquad \text{ for }p>1\text{ and }f \in L^p\text{ with }f\ge0.
}
Knopp \cite{Kno28} extend inequalities above to the case $p<0$.
With suitable substitution we can redefine these inequalities in term of power means, that is 
\begin{subequations}
\begin{align}
\sum_{n=1}^\infty \P_p(a_1,\dots,a_n)&\le C_p \sum_{n=1}^\infty a_n 
&\qquad \text{ for }a \in \ell^1\text{ with }a_n\ge0; \label{E:PowD}\\
\int_0^\infty \mathcal{P}_p \big(f|_{[0,x]}\big)dx &\le C_p \int_0^\infty f(x)\:dx
&\qquad \text{ for }f \in L^1\text{ with }f\ge0, \label{E:PowI}
\end{align}
\end{subequations}
where $\P_p$ is a (discrete) $p$-th Power mean and $\mathcal{P}_p$ is its integral counterpart. The value of $C_p$ equals
\Eq{*}{
C_p:=
\begin{cases} 
(1-p)^{-1/p}&p \in (-\infty,0) \cup (0,1), \\ 
e & p=0, \\
\end{cases} 
}
and these constants are sharp (in the limit case $p=0$ it is a result due to Carleman \cite{Car32}).  For the detailed history of this inequality we refer the reader to the book of Kufner-Maligranda-Persson \cite{KufMalPer07}. This approach was generalized by P\'ales-Persson \cite{PalPer04} who introduced the notion of (discrete) Hardy means by replacing the power mean in \eq{E:PowD} by an arbitrary mean $\M$ and constant $C_p$ by the sharp constant depending on $\M$, so-called \emph{Hardy constant} $\Hc(\M)\in(1,+\infty]$ (see section~\ref{sec:Hardy} for details). 
%
In the present paper we are going to generalize \eq{E:PowI} in the same way. To this end, we define integral means as a generalization of weighted means and deliver some extension-type results.
An important intermediate step is the weighted Hardy property, introduced recently in \cite{PalPas19a}. This issue can be sketched in the following way
\Eq{*}{
\text{discrete Hardy} \enspace\xrightarrow[]{established~in~2019}\enspace\text{weighted Hardy}\enspace\xrightarrow[]{aim~of~this~paper}\enspace\text{integral Hardy .}
}

In what follows, we  recall the notion of the weighted means. Later we introduce the new definition of the integral means as the  generalization of weighted ones.

At the  very end we study another property of means (i.e. distance between means) and prove that the (abstract) generalization of Deviation means coincides with the natural one. These results are also applied to the important subset of deviation means -- Gini means.

\section{Weighted means}
Definition of weighted mean first appeared in \cite{PalPas18b} in the process of reverse engineering. The aim was to generalize of so-called Kedlaya inequality (see also \cite{Ked94,Ked99}). This definition covers several particular cases of weighted means. We are going to recall a brief form of this definition. From now on $I \subset \R$ stands for an arbitrary interval and $R$ is an arbitrary subring of $\R$. For $n\in\N$ define the set of $n$-dimensional weight vectors
\Eq{*}{
  W_n(R):=\{(\lambda_1,\dots,\lambda_n)\in R^n\mid\lambda_1,\dots,\lambda_n\geq0,\,\lambda_1+\dots+\lambda_n>0\}.
}
An \emph{$R$-weighted mean on $I$} is a function 
$\M \colon \bigcup_{n=1}^{\infty} I^n \times W_n(R) \to I$ which satisfies four axioms: nullhomogeneous in the weights, reduction principle, mean value property and elimination principle. Following \cite{PalPas18b}, let us introduce elementary properties of these means. A weighted mean $\M$ is said to be \emph{symmetric}, if for all $n \in \N$, $x \in I^n$, $\lambda \in W_n(R)$, and a permutation $\sigma \in S_n$ we have $\M(x,\lambda) =\M(x\circ\sigma,\lambda\circ\sigma)$. 
Mean $\M$ is \emph{monotone} if it is nondecreasing in each of its entry. Furthermore $\M$ is \emph{convex} if for every $n \in \N$ and $\lambda \in W_n(R)$ the mapping $I^n \ni x \mapsto \M(x,\lambda) \in I$ is convex (or equivalently, by Bernstein-Doetsch theorem~\cite{BerDoe15}, Jensen convex). Similarly we introduce the definition of concavity. Furthermore we say that $\M$~is \emph{continuous in weights} if for all $n \in \N$ and $x \in I^n$ the mapping $\M(x,\cdot)$ is continuous on $W_n(R)$. Analogously $\M$ is \emph{continuous in entries} is for all $n \in \N$ and $\lambda \in W_n(R)$ the mapping $\M(\cdot,\lambda)$ is continuous on $I^n$. Naturally $\M$ is \emph{continuous in entries and weights} if it is continuous on each $I^n \times W_n(R)$ (as a multivariable function).

It can be proved that every $R$-weighted mean admit a unique extension to $R^*$-weighted mean ($R^*$ stands for the quotient field, i.e. the smallest field generated by $R$). Moreover this extension preserves all properties above (cf. \cite[Theorems 2.2--2.5]{PalPas18b}). Thus from now on we always assume that weighted means are defined on the fields.
Therefore every repetition invariant mean (being a $\Z$-weighted mean) can be extend to the $\Q$-weighted mean and, whenever there exists a continuous extension, to the $\R$-weighted mean. What is more, these extensions are uniquely determined and in most cases they coincide with  already known generalizations.


\subsection{$\K$-simple functions. Sum-type and integral-type notation}
As we are going to generalize weighted means to the integral ones, we recall certain function-type approach to weighted means from \cite{PalPas18b}. 

First we introduce so-called $\K$-intervals (here and below $\K$ is an arbitrary subfield of $\R$). We say that $D\subseteq\R$ is a \emph{$\K$-interval} if $D$ is of the form $[a,b)$ for some $a,b\in \K$. For a given $\K$-interval $D=[a,b)$, a function $f\colon D\to I$ is called \emph{$\K$-simple} if there exist a partition of $D$ into a finite number of $\K$-intervals $\{D_i\}_{i=1}^n$ such that:
\begin{enumerate}[(i)]
 \item $\sup D_i=\inf D_{i+1}$ for all $i\in\{1,\dots,n-1\}$, 
 \item $f$ is constant on each $D_i$.
 \end{enumerate} 
 
 In order to avoid some technical issues we extend this definition to all bounded intervals with endpoints in $\K$ in a natural way.  
 Then, for a $\K$-weighted mean $\M$ on $I$ and \mbox{$\K$-simple} function $f$ like above, we define
\Eq{E:int_notion}{
\Mm f(x) dx:=\M\big((f|_{D_1},\dots,f|_{D_n}),(|D_1|,\dots,|D_n|)\big),
}
where $|\cdot|$ stands for the Lebesgue measure of the set; furthermore whenever we say about measure we refer to the Lebesgue one.
Then $\M$ is monotone if and only if for every pair of $\K$-simple functions $f,g \colon D \to I$ with $f\le g$ the inequality $\M f(x)dx \le \M g(x)dx$ is valid.
This notion is natural in some sense as for the arithmetic mean (denoted by $\A$) and an $\R$-simple function $f \colon D \to \R$ we have
\Eq{*}{
\Ar f(x)dx=\frac{1}{|D|} \int_D f(x)\:dx=\fint_D f(x)\:dx.
}

To define the notion of symmetry let us first introduce equidistributed functions. We say that the pair $f_1 \colon D_1 \to \R$ and $f_2 \colon D_2 \to \R$ 
(where $D_1,D_2 \subseteq \R$)
of Lebesgue measurable functions is \emph{equidistributed} if 
$\abs{f_1^{-1}(U)}=\abs{f_2^{-1}(U)}$ for all open subintervals (or, equivalently, Borel subsets) $U$ of $\R$. One can show that $f_1,f_2$ as above are equidistributed if and only if
\Eq{*}{
\int_{D_1} \varphi(f_1(x))\:dx=\int_{D_2} \varphi(f_2(x))\:dx
}
 for every continuous, compactly supported test function $\varphi \colon I \to \R$.

 Observe that if $D_i$-s are bounded then, applying test function $\varphi \equiv 1$, we obtain that they have the same measure.

 \medskip

Using this notion we say that $\M $ is \emph{symmetric} if $\M f(x)dx=\M g(x)dx$ for every pair of equidistributed $\K$-simple functions $f,g \colon D \to I$. 
Moreover we keep the definite-integral-type convention, i.e. $\Mm_a^b f(x) dx:= \Mm f|_{[a,b)}(x)dx$ for $a,\,b \in \K \cap D$ with $a<b$.

\section{Functional and integral means}
First, we need to adapt the mean value property to the integral setting. For two intervals $I,\,J \subset \R$ let $\calL(J,I)$ be a family of all integrable functions $f \colon J \to I$. We abbreviate this notion to $\calL(I):=\bigcup_{J \text{ bounded}} \calL(J,I)$. Now for all $f\in \calL(\R)$ define $\RG(f)$ as the smallest closed interval containing the range of $f$. Then a \emph{functional mean on $I$} is a function $\iM \colon X \to I$ such that $X \subset \calL(I)$, and
$\iM(f) \in \RG(f)$ for every $f \in X$.  For a bounded functions we obviously have $\RG(f)=[\inf f,\sup f]$.

In what follows we introduce few properties of functional means. They are all adapted from the weighted counterpart. Functional mean $\iM$ is \emph{monotone} if for all $f,g \in X$ with equal domains such that $f\le g$ we have $\iM(f)\le \iM(g)$. We say that $\iM$ is \emph{convex} if $X$ is convex and for all $f,g \in X$ with equal domains and $\alpha \in [0,1]$ we have $\alpha f+(1-\alpha)g \in X$ and $\iM(\alpha f+(1-\alpha)g)\le \alpha \iM(f)+(1-\alpha) \iM(g)$; if this inequality is reversed then $\iM$ is \emph{concave}.

Let us now specify two more properties. They are slightly different then the one which were introduced in a weighted setting. We say that $\iM$ is an \emph{integral mean} (or \emph{integral-type mean}) if for all $f,g \in \calL(I)$ such that $f=g$ almost everywhere we have $\iM(f)=\iM(g)$. 

To illustrate the difference between the functional- and integral-type means let us mention that $\sup$ and $\inf$ are function-type while $\esssup$ and $\essinf$ are integral-type (obviously they are functional-type too).

Similarly to the weighted setting we say that an integral mean $\iM \colon X \to I$ is \emph{symmetric} if for every pair of equidistrubuted functions $f,\,g  \in \calL(I)$ we have $f \in X$ if and only if $g \in X$ and, moreover, $\iM(f)=\iM(g)$ (whenever these functions belong to $X$).

Now, based on \cite[section~2.8.2]{Duo01}, for every function  $f \in \calL(J,\R)$ such that $\RG(f)$ is bounded we define its \emph{decreasing rearrangement} $f^*\colon\big[0,|J|\big) \to \RG(f)$ by
\Eq{*}{
f^*(\tau):=\inf\big\{t\in \RG(f) \colon \big|\{x\colon f(x)>t\}\big|\leq\tau\big\}\qquad \text{ for }\tau \in [0,|J|).
}

Let us recall the main property of this transformation which arise from \cite{Duo01}.
\begin{lem}
 Let $I,J \subset \R$ be bounded intervals and $f \in \calL(J,I)$. Then $f^*$ is the unique nonincreasing and right continuous function on $[0,|J|)$ which is equidistributed with $f$.
\end{lem}

Now we are heading towards the following problem. Let $\M$ be a $\K$-weighted mean on $I$. Then, as we already proved, $\M$ generates an integral mean on a family of all $\K$-simple functions. We intent to extend this domain.
To this end, we modify the definition of the Lebesgue integral.

\subsection{$\K$-simple topology}

$I,\ J\subset \R$ be intervals such that $J$ is $\K$-simple.
For all $f \in \calL(J,I)$ and $\varepsilon>0$ we define the set
\Eq{*}{
B_\varepsilon^\K(f):=\big\{ g \colon J \to I \colon g \text{ is a }\K\text{-simple function},\, | \RG(g) \setminus \RG(f)|<\varepsilon,\text{ and} \norm{f-g}_1 < \varepsilon\big\}.
}
Observe that for all $\varepsilon>0$ we have $f \in B_\varepsilon^\K(f)$ if and only if $f$ is $\K$-simple. 
Thus $B_\varepsilon^\K$ can be considered to be either a neighbourhood of a vicinity of $f$. In order to avoid this drawback, for a measurable function $f$ as above and $\varepsilon \in (0,+\infty)$ we define its $\varepsilon$-neighbourhood $\mathcal{B}_{\varepsilon}^\K(f):=\{f\} \cup B_\varepsilon^\K(f)$. Then for all $f \in \calL(I)$ the family $(\mathcal{B}_\varepsilon^\K(f))_{\varepsilon>0}$ is the local basis of some topology $\tau_\K$ on $\calL(J,I)$. 

For a fixed interval $I$ we sum up topologies $\tau_\K$ on $\calL(J,I)$ over all bounded subintervals $J$ of $\R$. This led us to the topology $\tau_\K$ on $\calL(I)$ defined as follows. For a given $f \in \calL(I)$ and a sequence $(g_n)$ having entries is $\calL(I)$ we have $g_n \to f$ in $\tau_\K $ if and only if all conditions below are satisfied:
\begin{enumerate}[1.]
 \item there exists $n_0\in\N$ such that $\Dom(g_n)=\Dom(f)$ for all $n\ge n_0$; 
  \item there exists $n_1\in\N$ such that $g_n$ is $\K$-simple for all $n\ge n_1$;
\item for all $\varepsilon>0$ there exists $p_\varepsilon\in \N$ such that $|\RG(g_n)\setminus\RG(f)|<\varepsilon$ for all $n>p_\varepsilon$;
\item for all $\varepsilon>0$ there exists $q_\varepsilon\in \N$ such that $\norm{g_n-f}_1<\varepsilon$ for all $n>q_\varepsilon$.
 \end{enumerate}

Obviously we can assume without loss of generality that $n_0=n_1$ and $p_\varepsilon=q_\varepsilon$ for all $\varepsilon>0$, as it is handy.
Furthermore, as the finite prefix does not affect the convergence of sequence, we can assume that $n_0=n_1=1$,
which is equivalent to the fact that all elements in the sequence are $\K$-simple functions defined on $\Dom(f)$.

Whenever the filed $\K$ is known we abbreviate these notions. More precisely we will denote briefly $\tau$ instead of $\tau_\K$, $B_\varepsilon(f)$ instead of $B_\varepsilon^\K(f)$ and $g \tauto f$ to denote the convergence to $f$ in the topology $\tau=\tau_\K$.

Let us now show few preliminary results concerning functional means. First we show that $B_\varepsilon$ is never empty.

\begin{lem}\label{lem:INT1}
Let $J\subset \R$ be a $\K$-simple interval, $I$ be a bounded interval, and $f \in \calL(J,I)$. Then for every $\varepsilon>0$ the set $B_\varepsilon^\K(f)$ is nonempty.
\end{lem}

\begin{proof}
Observe that the (uniquely determined) affine bijection $\phi \colon [0,1) \to J$ maps $\K$-simple intervals to $\K$-simple intervals (as well as its converse). Therefore we can replace $f \in \calL(J,I)$ by $f \circ \phi \in \calL([0,1),I)$ as there exists a positive constant $C$ such that 
\Eq{*}{
\{\phi^{-1}\circ g \colon g \in B^\K_\varepsilon(f\circ \phi )\} \subset B^\K_{C\varepsilon}(f)\qquad\text{ for all }\varepsilon>0\text{ and }f \in \calL(J,I).
}
Thus from now one we may assume that $J=[0,1)$.
Next, as $f\colon [0,1) \to I$ is Lebesgue integrable, there exists a bounded function $f_0 \colon [0,1) \to \RG(f)$ such that $\norm{f-f_0}_1<\tfrac\varepsilon4$. 

Next, by Luzin's theorem, there exists a set compact subset  $L \subset [0,1)$ such that $f_0|_L$ is continuous and $|L|>1-\frac{\varepsilon}{4\cdot\norm{f_0}_\infty}$. Then, by Tietze (Urysohn-Brouwer) extension theorem, there exists a continuous function $f_1 \colon [0,1] \to \RG(f_0)$ such that $f_0|_L=f_1|_L$.

Then $f_1$ being a continuous function defined on a compact interval is also uniformly continuous and bounded. In particular there exists $K \in \N$ such that 
\Eq{01122020a}{
\sup_{x \in \left[\frac{i}K,\frac {i+1}K\right)} f_1(x)-\inf_{x \in \left[\frac{i}K,\frac {i+1}K\right)} f_1(x) < \tfrac{\varepsilon}{4} \quad \text{ for all }i\in\{0,\dots,K-1\}.
}

Define a function $g \colon [0,1) \to \RG(f_1)$ by 
\Eq{*}{
g|_{\left[\frac{i}K,\frac {i+1}K\right)}=f_1\big(\tfrac{i}K\big)\qquad\text{ for all }i \in \{0,\dots,K-1\}.
}
Obviously $g$ is $\Q$-simple, and therefore $\K$-simple. We show that $g \in B_\varepsilon(f)$. 
Indeed, assertion \eq{01122020a} implies $\norm{f_1-g}_\infty<\frac\varepsilon4$. Thus, binding all properties above, we obtain
\Eq{*}{
\norm{f-f_0}_1&\le\tfrac\varepsilon4,\\
\norm{f_0-f_1}_1&\le|[0,1) \setminus L| \cdot |\RG(f_0)|< \tfrac\varepsilon{4\norm{f_0}_\infty}\cdot 2\norm{f_0}_\infty = \tfrac\varepsilon2,\\
\norm{f_1-g}_1&\le \norm{f_1-g}_\infty < \tfrac\varepsilon4,
}
which, by the triangle inequality, yields
$\norm{f-g}_1< \varepsilon$.
Moreover we have a series of inclusions $\RG (g) \subseteq \RG (f_1) \subseteq \RG (f_0) \subseteq \RG (f)$.
Thus $g \in B_\varepsilon^\K(f)$ and, as a trivial consequence, $B_\varepsilon^\K(f) \ne \emptyset$.
\end{proof}

We now show some sort of a triangle inequality for $B_\varepsilon$.
\begin{lem}[Triangle inequality]\label{lem:triangle}
Let $I,J \subset \R$ be intervals, $f \in \calL(J,I)$, and $\delta,\varepsilon \in (0,+\infty)$ and consider an extended neighborhood
\Eq{*}{
\beta_\varepsilon(f):=\big\{ g \in \calL(J,I) \colon | \RG(g) \setminus \RG(f)|<\varepsilon\text{ and } \norm{f-g}_1 < \varepsilon\big\}.
}
Then $B_\varepsilon^\K(f)\subset \beta_\varepsilon(f)$ and 
\Eq{E:TI}{
\bigcup_{g \in \beta_\varepsilon(f)} B_\delta^\K(g) \subseteq B_{\varepsilon+\delta}^\K(f).
}

\end{lem}
\begin{proof}
The inclusion $B_\varepsilon^\K(f)\subset \beta_\varepsilon(f)$ is obvious. 

Now take $g \in \beta_\varepsilon(f)$ and $h \in B_\delta^\K(g)$ arbitrarily. 
 Then $h$ is $\K$-simple and
 \Eq{*}{
 \big|\Rg(h) \setminus \Rg(f)\big|&\le 
 \big|\Rg(g) \setminus \Rg(f)\big|+ \big|\Rg(h) \setminus \Rg(g)\big| \le \varepsilon+\delta;\\
\norm{f -h}_1 &\le 
 \norm{f -g }_1+ \norm{g-h}_1 \le \varepsilon+\delta,\\
 }
 which proves $h \in B_{\varepsilon+\delta}(f)$. Since $h$ was an arbitrary element in $B_\delta(g)$ we get $B_\delta(g) \subset B_{\varepsilon+\delta}(f)$.
 Finally, as $g$ was an arbitrarily element in $\beta_\varepsilon(f)$, we obtain desired inclusion \eq{E:TI}.
\end{proof}

\subsection{Canonical embeddings} For a bounded interval $J \subset \R$ we define its \emph{cannonical embedding} $\kappa_J$ as the unique affine and strictly increasing function $\kappa_J \colon J \to [0,1]$ with $\kappa_J(J) \supset (0,1)$. Furthermore for a function $f$ define on a subinterval of $I$ we set $\kappa_f:=\kappa_{\Dom(f)}^{-1}$.

Then $\kappa_f$ in the bijection between a unit interval and $\Dom(f)$. Furthermore $\kappa_f$ is a homeomorphism.

\subsection{Lower and upper means}
By the definition almost all functions in a sequence which is convergent in $\tau=\tau_\K$ are $\K$-simple. Thus, for a $\K$-weighted mean $\M$ on $I$, we set functional means $\LM,\,\UM \colon \calL(I) \to I$ by
\Eq{*}{
\LM(f)&:=\liminf_{g \tauto f} \M g(x)dx= \lim_{\varepsilon \to 0} \:\:\inf \big\{ \M g(x)dx \colon g \in B_\varepsilon(f \circ \kappa_f)\big\}\:, \\
\UM(f)&:=\limsup_{g \tauto f} \M g(x)dx=\lim_{\varepsilon \to 0} \:\sup \big\{ \M g(x)dx \colon g \in B_\varepsilon(f \circ \kappa_f)\big\}\:. 
}

Observe that whenever $f$ is a $\K$-simple function then $f \circ \kappa_f \in B_\varepsilon(f\circ \kappa_f)$ for all $\varepsilon>0$, thus $\M^L(f) \le \M f(x)dx \le \M^U(f)$. Therefore whenever $\M^L(f)=\M^U(f)$ for some function $f \in \calL(I)$ then we say that $\M^0(f)$ is well-defined for a function $f$, and define $\M^0(f):=\LM(f)$ or, equivalently, $\M^0(f):=\UM(f)$. Then $\M^0(f)=\M f(x)dx$ for every $\K$-simple function $f$ such that $\M^0(f)$ is well-defined. Obviously $\M^0$ is continuous in the topology $\tau$ in its domain. It implies that if $\M^0$ is integral-type then it is also continuous in $L^1 \cap L^\infty$.

Furthermore, by the definition of $\kappa$, for $g:=f\circ \kappa_f$ we have $g=g\circ\kappa_g$, and thus the mapping $f \mapsto f \circ \kappa_f$ is a projection which maps $\calL(I)$ onto the sum of four sets, i.e. 
\Eq{Projset}{
\Pi(I):=\calL\big((0,1),I\big) \cup \calL\big((0,1],I\big) \cup \calL\big([0,1),I\big) \cup \calL\big([0,1],I\big).
}
Thus the following lemma is easy to see.   

\begin{lem}\label{lem:EQAff}
 Let $\M$ be a $\K$-weighted mean of $I$, $f \in \mathcal{L}(I)$ and $\xi\colon \R \to \R$ be a strictly increasing affine function. Then
$\M^L(f)=\M^L(f\circ \xi)$
 and $\M^U(f)=\M^U(f\circ \xi)$.
\end{lem}

Let us now show a simple example of these means. Its proof is easy and therefore omitted.
\begin{exa}
For all $f \in \calL(\R)$ we have
\Eq{*}{
{\min}^L(f)&=\inf(f)&\qquad
{\min}^U(f)&=\essinf(f)\\
{\max}^L(f)&=\esssup(f)&\qquad
{\max}^U(f)&=\sup(f)
}
\end{exa}

There is the serious issue beyond this problem. Namely we it is not true in general that if two functions are equal almost everywhere then corresponding lower (or upper) means are equal. On the other hand, our main goal is to generalize weighted means in the integral-type flavor. Luckily, it turns out that (in most cases) this problem has the same background -- i.e. the difference between ``regular'' and essential supremum or infimum. In order to avoid this problem we restrict our consideration to the family 
where this value are the same. This motivates us to define
\Eq{*}{
\calL_0(I):= \big\{ f\in \calL(I) \colon \esssup(f)=\sup(f)\text{ and }\essinf(f)=\inf(f)\big\}.
}

Now, for all $f\in \calL_0(I)$ we obviously have $\RG(f^*)=\RG(f)$. This property plays a key role in the symmetry of integral means. We also define $\calL_0(J,I):=\calL(J,I)\cap \calL_0(I)$.

Now we are in the position the establish the initial results for $\LM$ and $\UM$.
\begin{prop}
For every $\K$-weighted mean $\M$ on $I$ both $\LM$ and $\UM$ are functional means on $I$ and $\LM \le \UM$.
\end{prop}

\begin{proof}
 By Lemma~\ref{lem:INT1} we know that $A_\varepsilon:=\big\{ \M g(x)\:dx \colon g \in B_\varepsilon(f)\big\}$ is a nonempty subset of $\RG(f)$ for all $\varepsilon>0$. Therefore $\inf f \le \inf A_\varepsilon \le \sup A_\varepsilon \le  \sup f$. In the limit case $\varepsilon \to 0$ we obtain  $\inf f\le \LM(f)\le \UM(f)\le \sup f$. 
\end{proof}


Now we show some result concerning symmetry of means
\begin{lem}\label{lem:SymIN}
 Let $\M$ be a symmetric \mbox{$\K$-weighted} mean on $I$,
 and $f \in \calL_0(I)$. Then
\Eq{*}{
\M^L(f^*)\le\M^L(f)
\qquad \text { and } \qquad
\M^U(f^*)\ge\M^U(f).
}
In particular $\M^L(f^*)=\M^U(f^*)$ implies $\M^L(f)=\M^U(f)$.
\end{lem}
\begin{proof}
Take $f \in \Pi(I)$, $\varepsilon\in(0,+\infty)$ and $g \in B_\varepsilon(f)$. Then, as $g$ is $\K$-simple we obtain that so is its rearrangement $g^*$. Furthermore as the rearrangement preserves the range we get $| \RG(g^*)\setminus \RG(f^*)| <\varepsilon$. 

 Since monotone rearrangement is $1$-Lipschitz selfmapping of $L^2$,
in view of H\"older inequality, we get
 \Eq{*}{
\norm{f^*-g^*}_1^2\le 
\norm{1}_2^2 \norm{f^*-g^*}_2^2 \le  \norm{f-g}_2^2 \le \norm{f-g}_\infty \norm{f-g}_1\le |I|\varepsilon.
 }
 Thus $\norm{f^*-g^*}_1 \le \sqrt{|I|\varepsilon}$. Consequently $g^* \in B_{(|I|\varepsilon)^{1/2}}(f^*)$, and so
\Eq{*}{
\LM(f)
&= \lim_{\varepsilon \to 0} \:\:\inf \big\{ \M g(x)dx \colon g \in B_\varepsilon(f)\big\}
= \lim_{\varepsilon \to 0} \:\:\inf \big\{ \M g^*(x)dx \colon g \in B_\varepsilon(f)\big\}\\
&\ge \lim_{\varepsilon \to 0} \:\:\inf \big\{ \M g^*(x)dx \colon g^* \in B_{(|I|\varepsilon)^{1/2}}(f^*)\big\}\ge\LM(f^*).
}
The second inequality is analogous. Thus in view of $\M^L\le \M^U$ we get 
\Eq{*}{
\M^L(f^*)\le\M^L(f) \le \M^U(f) \le \M^U(f^*),
}
which easily implies the remaining part.
\end{proof}

We are now going to show the condition which is sufficient to guarantee that $\M^L=\M^U$.
To this end, recall that $X$ is \emph{deeply contained} in $Y$ if the closure of $X$ is a subset of the interior of $Y$; we denote it briefly as $X \Subset Y$.
Then we have
\Eq{*}{
\interior \calL(J,I)=\big\{f \in \calL(J,I) \colon \RG(f)\Subset I \big\}.
}

Remarkably, this corollary is not necessary true for in $I^n \times W_n(\K)$. For example $\max$ is a convex $\R$-weighted mean on $\R$ which is not continuous in its weights.

Now, for intervals $I,J \subset \R$ and $f,f_1,f_2,\dots \in \mathcal{L}(J,I)$ we define monotone convergence
\Eq{*}{
f_n \nearrow f :\iff f_n \to f \text{ pointwise and } f_n \le f_{n+1} \text{ for all }n \in \N;\\
f_n \searrow f :\iff f_n \to f \text{ pointwise and } f_n \ge f_{n+1} \text{ for all }n \in \N.\\
}
Using this convergence we prove that lower and upper means admit a sort of semicontinuity. 
\begin{lem} \label{lem:38}
 Let $\M$ be a symmetric, monotone, $\R$-weighted mean on $I$. 
 Then, for every intervals $J \subset \R$, $I_0 \Subset I$, both $\M^L$ and $\M^U$ are symmetric on $\calL_0(J,I_0)$ and
for a given sequence of nonincreasing functions $f,f_1,f_2,\dots \in \calL_0(J,I_0)$ we have 
 \begin{enumerate}
 \item if $f_n \nearrow f$  then $\lim\limits_{n\to \infty} \M^L (f_n) = \M^L(f)$;
 \item if $f_n \searrow f$ then $\lim\limits_{n\to \infty} \M^U (f_n)=\M^U(f)$.
 \end{enumerate}
 Moreover if $\inf_{x \in J} |f_n(x)-f(x)| >0$ for all $n\in\N$ then
\Eq{*}{
\lim\limits_{n\to \infty} \M^L (f_n)=\lim\limits_{n\to \infty} \M^U (f_n)=\M^L(f) \text{ whenever }f_n \nearrow f;\\
\lim\limits_{n\to \infty} \M^L (f_n)=\lim\limits_{n\to \infty} \M^U (f_n)=\M^U(f) \text{ whenever }f_n \searrow f.
}

%
\end{lem}
\begin{proof}
We prove only the case $f_n\nearrow f$ as the second one is completely analogous.


Obviously, as $\M$ is monotone we obtain
$\M^L(f_n)\le \M^L(f)$ for all $n \in \N$. We prove the inequality 
\Eq{AX1}{
\lim_{n\to \infty} \M^L(f_n)\ge \M^L(f).
}

Take $\alpha\in \N$ such that $I_1:=[\inf I_0 -2^{-\alpha},\sup I_0+2^{-\alpha}] \Subset I$ and a sequence of function $(g_n)_{n=1}^\infty$ from $\calL(J,I_1)$ given by $g_n:=f_n-\frac{1}{2^{n+\alpha}}$. Then we have $\M_L(f_n)\ge \M_L(g_n)$ for all $n \ge 1$ and $g_n \nearrow f$. Whence it sufficient to show \eq{AX1} with $(f_n)_{n=1}^\infty$ replaced by $(g_n)_{n=1}^\infty$. Thus (substituting $g_n$ to $f_n$) from now on one can additionally assume 
\Eq{XT}{
f_n+2^{-n-\alpha}\le f_{n+1} \text{ for all } n\in \N.
}
Now define the set of gaps
\Eq{*}{
\Gamma_i:=\big\{x_0 \in [0,1] \colon \limsup_{x \to x_0} f(x) - \liminf_{x \to x_0} f(x) \ge 2^{-i-\alpha} \big\}.
}

Since $f$ is nonincreasing and bounded, each $\Gamma_i$ is finite. Therefore the set $(0,1) \setminus \Gamma_i$  consists of finitely many open intervals, say $J_{i,1},\dots,J_{i,k_i}$. Then 
\Eq{*}{
\limsup_{x \to x_0} f(x) -2^{-i-\alpha} <
\liminf_{x \to x_0} f(x)\qquad \text{ for all }i\in\N\text{, }l\in \{1,\dots,k_i\}\text{ and }x \in J_{i,l}.
}

In particular each for each pair $(i,l)$ as above, set
\Eq{*}{
T_{i,l}:=\{(x,y)\in J_{i,l} \times I \colon f(x)-2^{-i-\alpha}< y < f(x)\}
}
 is a connected and open subset of $J_{i,l} \times [\inf I_0 - 2^{-i-\alpha},\sup I_0]$. Now we sum-up $T_{i,l}$ over $l$ and define single-parameter family
 \Eq{*}{
 T_i :=\cl \Big(\bigcup_{l=1}^{k_i} T_{i,l} \Big) \subset [0,1] \times [\inf I_0 - 2^{-i-\alpha}, \sup I_0] \subset [0,1] \times I\qquad (i \in\N).
 }\

Furthermore since the mapping $n \mapsto f_n$ is nondecreasing and each $f_n$ is monotone we obtain (adapting the classical P\'olya result \cite{Pol20}) that for all $i \in \N$ and $l \in\{1,\dots,k_i\}$ there exists $n_{i,l}\in\N$ such that $\Graph(f_n|_{J_{i,l}}) \subset T_{i,l}$ for all $n>n_{i,l}$.
Thus with $n_i:=\max(n_{i,1},\dots,n_{i,k_i})$ we obtain $\Graph(f_n) \subset T_i$  for all $i \in \N$ and $n>n_i$. Consequently, $f_n \in \beta_{2^{-i-\alpha}}(f)$ for all $n>n_i$. Thus by the triangle inequality expressed in Lemma~\ref{lem:triangle} we get
\Eq{*}{
B_{2^{-i-\alpha}}(f_n) \subseteq B_{2^{-i-\alpha+1}}(f) \qquad \text{ for all }i \in \N\text{ and }n>n_i,
}
which by the definition of $\M^L$ implies \eq{AX1} and completes the proof of the first condition (1). The first condition (2) is analogous.

Now we proceed to the moreover part. More precisely we show that, under the additional assumption, the inequality 
\Eq{DX1}{
\lim_{n \to \infty} \M^U(f_n)\le \M^L(f)
}
holds. To this end, fix $n \in \N$. Then there exists $\delta\in(0,+\infty)$ such that $f_n+3\delta\le f$. Let 
\Eq{*}{
Q:=\big\{(x,y) \in J \times I \colon f(x)-2\delta< y <f(x)-\delta\big\}.
}
Repeating the same argumentation as above, one can split $J$ to finitely many intervals $J_1,\dots,J_k$ such that each $Q \cap (J_i \times I)$ is connected. In particular there exists a function $g \colon J \to I$ such that $\Graph(g)\subset Q$ and $g$ is continuous on each $J_i$.

Since the family of simple functions is dense in a family of monotone functions, one can take an $\R$-simple function $\chi_n \colon J \to I$ such that $\norm{\chi_n-g}_\infty \le \delta$. Then 
\Eq{*}
{
\chi_n \le g+\delta\le (f-\delta)+\delta= f \quad\text{ and }\quad\chi_n \ge g-\delta\ge f-3\delta\ge f_n.
}

Thus, for every $n \in \N$ one can take an $\R$-simple function $\chi_n \colon J \to I$ with $f_n \le \chi_n\le f$. Then we have $\M^U(f_n)\le \M\chi_n(x)\:dx \le \M^L(f)$. In a limit case as $n$ tends to infinity this inequality yields \eq{DX1}. Binding \eq{AX1}, \eq{DX1} and the trivial property $\M^L \le \M^U$ we obtain the moreover part.

Finally we need to show that $\M^L$ and $\M^U$ are symmetric. To this end take a function $f \in \calL_0(J,I_0)$ arbitrarily. Then there exists a sequence $(f_n)_{n=1}^\infty$ of function in $\calL(J,I)$ such that $f_n \nearrow f$ and $\inf |f_n-f| \cdot \inf |f_n^*-f^*|>0$ for all $n \in \N$. Then Lemma~\ref{lem:SymIN} yields
\Eq{*}{
\lim\limits_{n\to \infty} \M^L (f_n^*)\le \lim\limits_{n\to \infty} \M^L (f_n)\le\lim\limits_{n\to \infty} \M^U (f_n) \le\lim\limits_{n\to \infty} \M^U (f_n^*),
}
which shows that $\M^L(f^*)\le \M^L(f)\le \M^L(f^*)$, i.e. $\M^L$ is symmetric on $\calL_0(J,I_0)$. The proof for $\M^U$ is analogous.
\end{proof}

Let us put the reader attention to the following important fact.
\begin{rem}
 The assumption $\inf_{x \in J} |f_n(x)-f(x)| >0$ for all $n\in\N$ in the moreover part of Lemma~\ref{lem:38} cannot be omitted unless $\M^L=\M^U$. Indeed, applying the moreover part to the constant sequence $f_n := f\in \calL(I)$ we immediately obtain $\M^L(f)=\M^U(f)$.
\end{rem}

\subsection{Sufficient conditions to define $\M^0$}
In this section we establish few sufficient conditions to validate the equality $\M^L=\M^U$. Our initial result is very easy in view of the definition of $\M^0$.
\begin{prop}\label{prop:suff}
 Let $\M$ be a $\K$-weighted mean on $I$. Assume that there exists a subset $X\subset \calL(I)$ which is open in $\tau_\K$, an integral mean $\Nm \colon X \to I$ which is continuous in $\tau_\K$ and
\Eq{*}{
\Nm(\chi)=\Mm\chi(x)\:dx\text{ for every }\K\text{-simple function }\chi \in X.
}
Then the mean $\M^0$ is well-defined on $X$ and $\M^0|_X=\Nm$.
%
\end{prop}

It is useful in a several particular cases. In order to apply it let us define a set \Eq{*}{
\calL_C(I):=\bigcup_{\substack{S \Subset I\\S \text{ interval}}} \calL_0(S).
} 

Adapting the well-known fact that $\calL(J,I)$ is a convex set and the convex (or concave) function is continuous in the interior of its domain, we can show the sufficient condition to validate the equality $\M_L=\M_U$ for convex (or concave) means.

\begin{thm}
 Let $\M$ be a symmetric, monotone and convex (concave) $\R$-weighted mean on $I$. Then $\M^0$ is a functional-type mean which is well-defined on $\calL_C(I)$ which is convex (concave) in its domain.
\end{thm}
\begin{proof}
We show that for all $f \in \calL_C(I)$ the equality $\M^L(f)=\M^U(f)$ holds. We obviously have $\M^L(f)\le\M^U(f)$. 
For the converse inequality by Lemma~\ref{lem:SymIN}, we can restrict our consideration to nonincreasing functions only.
Furthermore by the definition of $\M^L$ an $\M^U$ we can take a function in the set $\Pi(I)$ defined in \eq{Projset}.
Thus let $f \in \Pi(I)$ be an arbitrary nondecreasing function such that $\RG (f) \Subset I$. 

Take a family of $\R$-simple functions $\chi_n \colon [0,1) \to I$ such that 
$f-\frac{1}{2^n} \le \chi_n \le f$.
 Furthermore, there exists a family $\eta_n\colon [0,1) \to I$ of $\R$-simple functions such that $f+\frac{1}{2^n}\le \eta_n \le f+\frac{1}{2^{n-1}}$. Then by Lemma~\ref{lem:38}
 \Eq{*}{
 \lim_{n \to \infty} \M \eta_n(x)\:dx = \M^U(f) \text{ and }  \lim_{n \to \infty} \M \chi_n(x)\:dx = \M^L(f).
 }
Applying this lemma again, as $\M$ is convex, one gets
 \Eq{*}{
\M^U(f)= \lim_{n \to \infty} \Mm \frac{\eta_n(x)+\chi_n(x)}2\:dx \le 
\lim_{n \to \infty} \frac{\M\eta_n(x)dx+\M\chi_n(x)dx}2=\frac{\M^U(f)+\M^L(f)}2.
 }
Thus $\M^L(f)\ge \M^U(f)$. As the converse inequality is trivial we obtain $\M^L(f)= \M^U(f)$.
 
\medskip

To show the convexity of $\M^0$ let $\alpha\in (0,1)$ and $f,g \in \Pi(I)$ be two functions for which $\M^0$ is well-defined. Then for all $\varepsilon>0$ there exist $f_\varepsilon \in B_{\varepsilon}(f)$ and $g_\varepsilon \in B_\varepsilon(g)$ such that 
$\Mm f_{\varepsilon}(x)\:dx<\M^0(f)+\varepsilon$ and $\Mm g_{\varepsilon}(x)\:dx<\M^0(g)+\varepsilon$. Thus, as $\M$ is convex and both $f_\varepsilon$, $g_\varepsilon$ are $\R$-simple functions we have, for all $\varepsilon>0$,
\Eq{*}{
\Mm \alpha f_\varepsilon(x)+(1-\alpha)g_\varepsilon(x)\:dx 
&\le \alpha\Mm f_{\varepsilon}(x)\:dx+(1-\alpha)\Mm g_{\varepsilon}(x)\:dx\\
&<\alpha\M^0(f)+(1-\alpha)\M^0(g)+\varepsilon.
}
But $\alpha f_\varepsilon+(1-\alpha)g_\varepsilon \tauto \alpha f+(1-\alpha)g$ so if we take the limit as $\varepsilon \to 0$ side-by-side we obtain
$\M^0(\alpha f+(1-\alpha )g) \le \alpha\M^0(f)+(1-\alpha)\M^0(g)$. Thus $\M^0$ is convex. 

The case when $\M$ is concave is completely analogous. 
\end{proof}

This theorem has an important corollary
\begin{cor}
Let $\M$ be a symmetric, monotone $\R$-weighted mean on $I$ which is either convex or concave. Then $\M$ is continuous on each $I^n \times (0,+\infty)^n$, where $n\in\N$.
\end{cor}

\section{Applications}
\subsection{Deviation means}
At the moment we apply our results to the family of deviation (Dar\'oczy) means \cite{Dar71b,Dar72b}. They are parameterized by a bivariate function. More precisely a function $E\colon I\times I\to\R$ is said to be a \emph{deviation} if
\begin{enumerate}[(a)]
\item  $E(x,x)=0$ for all $x\in I$,
\item for all $x\in I$, the map $I \ni y\mapsto E(x,y)$ is continuous and strictly decreasing.
\end{enumerate}

Then, for a deviation $E$, a number $n\in\N$, and a pair $(x,\lambda)\in I^n\times W_n(\R)$, we define a (weighted) \emph{deviation mean} $\D_E(x,\lambda)$ as the unique solution $y$ of the equation 
\Eq{*}{
\lambda_1 E(x_1,y)+\cdots+\lambda_n E(x_n,y)=0.
}

We restrict our consideration to continuous deviations only (i.e. $E$ is continuous as a bivariate function). Then $\D_E$ is the $\R$-weighted mean on $I$ which is continuous in entries and weights. 
Furthermore it is the well-known folk result that $\D_E$ is monotone if and only if $E$ is strictly increasing in its first variable. Using these facts one can formulate certain sufficient condition to guarantee that $\D_E^0$ is well-defined.

\begin{thm}\label{thm:DevExt}
Let $E \colon I \times I \to \R$ be a continuous deviation. Then $\D_E^0 \colon \calL_C(I) \to I$ is well defined and $\D_E^0(f)$ is the unique solution $y$ of the equation 
\Eq{*}{
\int E(f(t),y)\:dt=0 \qquad \text{ for all }f \in \calL_C(I).
}
Moreover $\D_E^0$ can be extended to $\calL(I)$ provided such a continuous extension exists.
\end{thm}

\begin{proof} We aim to use Proposition~\ref{prop:suff}.
For $f \in \Pi(I)$ define $e_f \colon I \to \R$ by
\Eq{*}{
e_f(y):=\int_0^1 E(f(t),y)\:dt.
}
First we prove that $e_f$ is continuous for all $f \in \Pi(I)\cap \calL_C(I)$. Indeed, as $E$ is continuous on the compact set $\RG(f)\times \RG(f)$, it is also uniformly continuous. Thus for all $\varepsilon>0$ there exists $\delta>0$ such that 
\Eq{*}{
\abs{E(x,y)-E(x,y')}<\varepsilon \text{ for all }x,y,y' \in \RG(f) \text{ with }\abs{y-y'}<\delta.
}
Then we simply have $\abs{e_f(y)-e_f(y')}\le\varepsilon$ for all $y,\ y'\in \RG(f)$ as above.

Next we show that $e_f$ is strictly decreasing for all $f$ as above. To this end take $y,y' \in \RG(f)$ with $y'>y$. As $E$ is continuous and strictly decreasing in the second variable and $\RG(f)$ is compact, there exists $\alpha>0$ such that 
\Eq{*}{E(x,y')\le E(x,y)-\frac{\alpha}{\abs{\Dom f}}\qquad\text{ for all } x \in \RG(f).}
Then we simply obtain $e_f(y')\le e_f(y)-\alpha<e_f(y)$.
Furthermore whenever $f \in \calL_C(I)$ we can easily verify that 
$e_f(\inf f) \ge 0$ and $e_f(\sup f) \le 0$.

Thus there exists a unique number $\Nm(f) \in \RG(f)$ such that $e_f(\Nm(f))=0$. 

\smallskip
In the next step we show that $\D_E^0(f)=\Nm(f)$ for all $f \in \calL_C([0,1),I)$. Equivalently, by the definition we need to show that $\Nm$ is continuous in $\tau=\tau_\R$.

To this end, take a sequence of $(\chi_n)_{n=1}^\infty$ with $\chi_n\tauto f$.

Then, since $f\in \calL_C([0,1),I)$, there exists a subinterval $J \Subset I$ such that almost all $\chi_n$ have values in $J$. We can omit the beginning of the sequence $\chi_n$ and assume all $\chi_n$-s are $\R$-simple functions in $\calL([0,1),J)$.

Fix $\varepsilon>0$. Since $E$ is uniformly continuous on $J\times J$, there exists $\delta:=\delta(\varepsilon)>0$ such that
\Eq{*}{
\abs{E(x,y)-E(x_0,y_0)}<\varepsilon \qquad \text{ for all }x,y,x_0,y_0\in J \text{ with }\abs{x-x_0}+\abs{y-y_0}<2\delta.
}
Set $D_n:=\{x \in [0,1)\colon |\chi_n(x)-f(x)|>\delta\}$. Then, there exists a number $n_\varepsilon \in \N$ such that 
$|D_n|<\varepsilon$ for all $n>n_\varepsilon$. Consequently for all $n>n_\varepsilon$ and $y \in J$ we have
\Eq{*}{
\abs{e_{\chi_n}(y)-e_f(y)}&=\abs{\int_0^1 E(\chi_n(t),y)-E(f(t),y)\:dt}\\
&\le\abs{\int_{D_n} E(\chi_n(t),y)-E(f(t),y)\:dt}+\abs{\int_{[0,1)\setminus D_n} E(\chi_n(t),y)-E(f(t),y)\:dt}\\
&\le|D_n| \sup_{x,y\in J} 2|E(x,y)| +\varepsilon \le \big( 1+ 2 \norm{E}_{L^\infty(J\times J)}\big) \,\varepsilon.
}
Thus the sequence $(e_{\chi_n})_{n=1}^\infty$ converges to $e_f$ uniformly on $J$.
Consequently as each $e_f$ is strictly decreasing we obtain
\Eq{*}{
\lim_{n \to \infty } e_{\chi_n}(y)=e_f(y) >0 \text{ for all }y < \Nm(f).
}
Now, as each $e_{\chi_n}$ is strictly decreasing and $\D_E\chi_n(x)\:dx$ is its unique zero, it yields
\Eq{*}{
\liminf_{n \to \infty} \D_E\chi_n(x)\:dx \ge \Nm(f).
}
Similarly we obtain the converse inequality. Thus the sequence $(\D_E\chi_n(x)\:dx)_{n=1}^\infty$ is convergent and 
\Eq{*}{
\liminf_{n \to \infty} \D_E\chi_n(x)\:dx = \Nm(f).
}
Finally, as $\chi_n$ was an arbitrary sequence which is convergent to $f$ in $\tau$ we get $\D_E^0(f) = \Nm(f)$.

The moreover part is a straightforward implication of the definition of $\D_E^0$.
\end{proof}

\subsection{\label{sec:Hardy} Hardy property} In this section we deliver a brief introduction to the Hardy property and its extension to integral means. This introduction is inspired by \cite{PalPas19b}, especially Theorem 2.8 therein. Let $\M$ be a symmetric and monotone $\K$-weighted mean on $I$ with $\inf I=0$. We define a Hardy constant $\Hc(\M)$ as the smallest extended real number such that
\Eq{*}{
\sum_{n=1}^\infty \lambda_n \M \big((x_1,\dots,x_n),(\lambda_1,\dots,\lambda_n)) \le \Hc(\M) \sum_{n=1}^\infty \lambda_nx_n
}
for every all-positive-elements sequences $(x_n)_{n=1}^\infty$  and $(\lambda_n)_{n=1}^\infty$ of elements in $I$ and $\K$, respectively.

Now we can introduce the integral counterpart of this notion. Namely, for a functional mean $\iM$ on $I$ we define a Hardy constant $\iHc(\iM)$ as the smallest extended real number such that
\Eq{*}{
\int_0^\infty \iM \big(f\big|_{[0,t)}\big)\:dt \le \iHc(\iM) \int_0^\infty f(t)\:dt
}
for every integrable function $f \colon [0,+\infty) \to I$. 

Let us first show an important lemma which shows that this value is suitable for finite integrals too.

\begin{lem}\label{lem:Hfininf}
 Let $\iM$ be a nullhomogeneous functional mean on $I$, i.e. $\iM(f(\alpha \,\cdot\,))=\iM(f)$ for all admissible $f$ and $\alpha>0$. Furthermore fix $s\in(0,+\infty)$. Then $\iHc(\iM)$ is the smallest extended real number $H$ such that
 \Eq{020121a}{
 \int_0^s \iM \big(f\big|_{[0,t)}\big)\:dt \le H \int_0^s f(t)\:dt\text{ for all }f\in \calL\big([0,s),I\big).
 }
\end{lem}
\begin{proof}
Let $H_0$ be the smallest extended real number such that \eq{020121a} holds with $H=H_0$.
We show both inequalities $H_0 \le \iHc(\iM)$ and $H_0 \ge \iHc(\iM)$. 

First let $s\in(0,+\infty)$, $f \in \calL\big([0,s),I\big)$ and $\varepsilon>0$ such that $\varepsilon e^{-s}\in I$. Define $g \colon [0,+\infty)\to I$ by 
\Eq{*}{
g(t):=\begin{cases}
       f(t) &\text{ for }t<s,\\
       \varepsilon e^{-t} &\text{ for }t\ge s.
      \end{cases}
}
Then we have
\Eq{*}{
\int_0^s \iM \big(f\big|_{[0,t)}\big)\:dt
&=\int_0^s \iM \big(g\big|_{[0,t)}\big)\:dt
\le \int_0^\infty \iM \big(g\big|_{[0,t)}\big)\:dt
\le \iHc(\iM) \int_0^\infty g(t)\:dt\\
&= \iHc(\iM) \big( \int_0^s f(t)\:dt+\varepsilon \int_s^\infty e^{-t}\big)
\le\iHc(\iM) \big( \int_0^s f(t)\:dt+\varepsilon\big).\\
}
Upon passing $\varepsilon\to 0$ we get 
\Eq{*}{
\int_0^s \iM \big(f\big|_{[0,t)}\big)\:dt \le\iHc(\iM) \int_0^s f(t)\:dt.
}
As $f$ was taken arbitrarily, it implies $H_0\le \iHc(\iM)$.

To show the converse inequality assume to the contrary that $H_0<\iHc(\iM)$, i.e. there exists an integrable function $f\colon [0,+\infty) \to I$ such that 
 \Eq{*}{
 \int_0^\infty \iM \big(f\big|_{[0,t)}\big)\:dt > H_0 \int_0^\infty f(t)\:dt.
 }
Then there exists $r \in (0,+\infty)$ which validates the inequality
 \Eq{*}{
 \int_0^r \iM \big(f\big|_{[0,t)}\big)\:dt > H_0 \int_0^r f(t)\:dt.
 }
Then we have
\Eq{*}{
H_0 \int_0^s f(\tfrac rs t)\:dt
&=H_0\tfrac{s}r \int_0^r f(t)\:dt
< \tfrac sr \int_0^r \iM \big(f\big|_{[0,t)}\big)\:dt\\
&=\tfrac sr \int_0^r \iM \big(f\big(\tfrac rs x \big)\big|_{x\in\big[0,\tfrac{st}r\big)}\big)\:dt
=\int_0^s \iM \big(f\big(\tfrac rs x \big)\big|_{x\in[0,t)}\big)\:dt
\le H_0 \int_0^s f(\tfrac rs t)\:dt,
}
which leads to a contradiction and completes the proof.
\end{proof}

Next, let us recall one lemma
\begin{lem}[\cite{PalPas19a}, Lemma~5.1]
 \label{lem:*}
 Let $\M$ be a $\K$-weighted, monotone mean on $I$ and $a \in \K \cap (0,\infty)$. Then, for any nonincreasing $\K$-simple function $f\colon [0,a) \to I$, the mapping $F\colon \K\cap(0,a] \to I$ given by $F(u):=\Mm_0^uf(t)\:dt$, is nonincreasing.
\end{lem}

Now we are in the position to prove the main result of this section.

\begin{thm}\label{thm:Hardy}
 Let $\M$ be a symmetric and monotone $\K$-weighted mean on $I$ with $\inf I=0$. 
Then $\iHc(\M^L)=\iHc(\M^U)=\Hc(\M)$.
 \end{thm}
\begin{proof}
 By $\M^L \le \M^U$ we have $\iHc(\M^L)\le\iHc(\M^U)$. We prove two remaining inequalities: (i)~$\Hc(\M) \le \iHc(\M^L)$ and (ii) 
$\iHc(\M^U) \le \Hc(\M)$. This naturally splits the whole proof into separate parts.

{\bf Part (i).} Fix $H\in[1,\Hc(\M))$. By the definition (see \cite[Theorem 2.8]{PalPas19b}) there exists a sequence $(\lambda_n)$
having elements in $\K_+:=\K\cap(0,+\infty)$ with $\sum_{n=1}^\infty \lambda_n=\infty$ and $(x_n)\in \ell^1(\lambda)$ such that 
\Eq{H1}{
\sum_{n=1}^\infty \lambda_n \M \big((x_1,\dots,x_n),(\lambda_1,\dots,\lambda_n)\big) \ge H \sum_{n=1}^\infty \lambda_nx_n.
}
Define $\Lambda_N:=\sum_{n=1}^N \lambda_n$ for $N\in\N$, $\Lambda_0:=0$, and $\chi \colon [0,\infty) \to I$ by
\Eq{*}{
\chi(t)=x_n \text{ for }t \in [\Lambda_{n-1},\Lambda_n)\text{ where }n\in\N.
}

Using this function we can rewrite \eq{H1} in the integral-type spirit as follows
\Eq{22dec1400}{
\sum_{n=1}^\infty \lambda_n \Mm_0^{\Lambda_n} \chi(t)\:dt \ge H \int_0^{\infty} \chi(t)\:dt.
}

But, as $\M$ is symmetric and monotone we obtain
\Eq{*}{
\Mm_0^u \chi(t)\:dt= \Mm_0^u \chi|_{[0,u)}(t)\:dt =\Mm_0^u (\chi|_{[0,u)})^*(t)\:dt 
\le \Mm_0^u \chi^*(t)\:dt \quad \text{ for all }u \in \K_+.
}

Now, for $\varepsilon>0$ set $f_\varepsilon \colon [0,+\infty)\to I$ by $f_\varepsilon(x) := \min\{\chi^*(x)+\varepsilon e^{-x},\sup I\}$. Then as $\M$ is monotone and $\chi^*$ is a $\K$-simple function below $f_\varepsilon$ we get
\Eq{*}{
\M^L \big(f_\varepsilon|_{[0,u)}\big) \ge \Mm_0^{u} \chi^*(t)\:dt \ge \Mm_0^{u} \chi(t)\:dt \qquad \text{ for all }u \in \K_+.
}
Thus by \eq{22dec1400}
\Eq{*}{
\sum_{n=1}^\infty \lambda_n \M^L \big(f_\varepsilon|_{[0,\Lambda_n)}\big) \ge H \int_0^{\infty} \chi(t)\:dt.
}
In view of Lemma~\ref{lem:EQAff}, as $f_\varepsilon$ is nonincreasing and $\M^L$ is monotone, we obtain 
\Eq{*}{
\M^L\big(f_\varepsilon|_{[0,u)}\big)=\M^L \big(f_\varepsilon(\tfrac{u}{u'}\:\cdot\:)|_{[0,u')}\big)\le \M^L \big(f_\varepsilon|_{[0,u')}\big) \qquad\text{ for all }u,u' \in \R\text{ with }u>u'>0,
}
i.e. the mapping $u \mapsto \M^L(f_\varepsilon|_{[0,u)})$ is nonincreasing. Therefore
\Eq{*}{
\sum_{n=1}^\infty \lambda_n \M^L(f_\varepsilon|_{[0,\Lambda_n)})  
\le \sum_{n=1}^\infty \int_{\Lambda_{n-1}}^{\Lambda_n} \M^L(f_\varepsilon|_{[0,u)})\:du   
=\int_0^{\infty} \M^L(f_\varepsilon|_{[0,u)})\:du.
}
This implies
\Eq{*}{
H \int_0^{\infty} \chi(t)\:dt 
&\le \sum_{n=1}^\infty \lambda_n \M^L(f_\varepsilon|_{[0,\Lambda_n)})\\
&\le  \int_0^{\infty} \M^L(f_\varepsilon|_{[0,u)})\:du 
\le \iHc(\M^L) \int_0^\infty f_\varepsilon(t)\:dt\\
&=\iHc(\M^L) \int_0^\infty \chi^*(t)+\varepsilon e^{-t}\:dt 
=\iHc(\M^L) \Big(\varepsilon+\int_0^\infty \chi(t)\:dt\Big),
}
in a limit case as $\varepsilon \to 0$, it implies $H \le \iHc(\M^L)$. Therefore, as $H\in[1,\Hc(\M))$ was taken arbitrarily, we obtain $\Hc(\M) \le \iHc(\M^L)$.

{\bf Part (ii).} Fix $K\in\N$ and let $f \colon [0,1) \to I$ be a measurable function. 
Repeating the argumentation above by Lemma~\ref{lem:SymIN} we get

\Eq{*}{
\M^U\big(f|_{[0,q)}\big) \le \M^U\big((f|_{[0,q)})^*\big)\le\M^U(f^*|_{[0,q)}) \text{ for all }q \in\K \cap (0,1].
}

On the other hand, there exists a nonincreasing function $g \in B_{1/K}(f^*)$ with $g \ge f^*$.  Whence, by the definition of $\M^U$ for all $q \in \K \cap (0,1]$ we have
\Eq{*}{
\M^U\big(f^*\big|_{[0,q)}\big) \le \Mm_0^q g(u)\:du.
} 
Upon passing $q \nearrow t$ we obtain $f^*(\frac qt (\cdot)) \searrow f^*(\cdot)$ and thus, by Lemma~\ref{lem:38},
\Eq{*}{
\M^U\big(f\big|_{[0,t)}\big)\le\M^U\big(f^*\big|_{[0,t)}\big)\le\limsup_{\substack{q \nearrow t\\q \in \K}} \Mm_0^{q} g(u)\:du\qquad \text{ for all }t \in[0,1).
}
But as $g$ is nonincreasing, in view of Lemma~\ref{lem:*}, so is $\K \ni q\mapsto \Mm_0^{q} g(t)\:dt$. Thus
\Eq{*}{
\M^U\big(f\big|_{[0,t)}\big)\le\inf_{q \in [0,t)\cap\K} \Mm_0^{q} g(u)\:du\qquad \text{ for all }t \in[0,1).
}

Applying the upper-Riemann integral side-by-side we obtain
\Eq{*}{
\upRiemannint01 \M^U\big(f\big|_{[0,t)}\big)\:dt \le \upRiemannint01 \inf_{q \in [0,t)\cap\K} \Mm_0^{v} g(u)\:du\:dt\:.
}
Define $h \colon [0,1) \to I$ by 
\Eq{*}{
h|_{\big[\frac{i-1}{K},\frac{i}{K}\big)}=g\big(\tfrac{i-1}{K}\big)\qquad\text{ for all }i \in \{1,\dots,K\}.
}
Then $h$ is nonincreasing and $h\ge g$. Thus
\Eq{*}{
\upRiemannint{0}{1} \M^U\big(f\big|_{[0,t)}\big)\:dt
&\le \upRiemannint{0}{1} \inf_{q \in [0,t)\cap\K} \Mm_0^{q} g(u)\:du\:dt
\le\upRiemannint{0}{1} \inf_{q \in [0,t)\cap\K} \Mm_0^{q} h(t)\:du\:dt\\
&=\sum_{n=1}^K \upRiemannint{\frac{n-1}K}{\frac nK} \inf_{q \in [0,t)\cap\K} \Mm_0^{q} h(t)\:du\:dt
\le \frac{h(0)}K+\frac1K\sum_{n=2}^K \Mm_0^{\frac{n-1}K} h(t)\:dt
\\
&=\frac{1}{K}\sum_{n=1}^K\M\Big(h(0),h\big(\tfrac1K\big),\dots,h\big(\tfrac{n-1}K\big)\Big)
\le \frac{\Hc(\M)}{K}\sum_{n=1}^Kh\big(\tfrac{n-1}K\big)\\
&\le\Hc(\M) \Big( \frac{h(0)}{K}+\int_0^1  g(x)dx\Big).
}
But as $g \in B_{1/K}(f^*)$ we have $\norm{f^*-g}_1\le \tfrac1K$. Moreover $h(0)=g(0)$, whence
\Eq{*}{
\upRiemannint01 \M^U\big(f\big|_{[0,u)}\big) &\le \Hc(\M) \bigg( \frac{g(0)+1}{K}+\int_0^1  f^*(x)dx\bigg)
=\Hc(\M) \bigg( \frac{g(0)+1}{K}+\int_0^1  f(x)dx\bigg).
}
If we take a limit $K \to \infty$ we get
\Eq{*}{
\upRiemannint01 \M^U\big(f\big|_{[0,u)}\big) &\le \Hc(\M) \int_0^1  f(x)dx.
}
By Lemma~\ref{lem:Hfininf} this inequality implies $\iHc(\M^U) \le \Hc(\M)$, which completes the proof.
\end{proof}

Let us now show the important corollary. Namely, for $p,q \in \R$ we can define the function $\chi_{p,q}\colon\R_+\to\R$ by
\Eq{*}{
  \chi_{p,q}(x)
  :=\begin{cases}
    \dfrac{x^p-x^q}{p-q} & \mbox{ if } p\neq q, \\[4mm]
    x^p\ln(x) & \mbox{ if } p=q.
    \end{cases}
}
Then for every $p,q \in \R$, function $E_{p,q}\colon \R_+^2\to\R$ defined by 
\Eq{*}{
  E_{p,q}(x,y):=y^p\chi_{p,q}\Big(\frac{x}{y}\Big)
}
is a deviation function on $\R_+$. The weighted deviation mean generated by $E_{p,q}$ will be denoted by $\G_{p,q}$ and called the \emph{Gini mean (of parameter $p,q$)} (cf.\ \cite{Gin38}). Mean $\G_{p,q}$ has the following explicit form:
\Eq{GM}{
  \G_{p,q}(x,\lambda)
   :=\left\{\begin{array}{ll}
    \left(\dfrac{\lambda_1x_1^p+\cdots+\lambda_nx_n^p}
           {\lambda_1x_1^q+\cdots+\lambda_nx_n^q}\right)^{\frac{1}{p-q}} 
      &\mbox{if }p\neq q, \\[4mm]
     \exp\left(\dfrac{\lambda_1x_1^p\ln(x_1)+\cdots+\lambda_nx_n^p\ln(x_n)}
           {\lambda_1x_1^p+\cdots+\lambda_nx_n^p}\right) \quad
      &\mbox{if }p=q.
    \end{array}\right.
}
Clearly, in the particular case $q=0$, the mean $\G_{p,q}$ reduces to the $p$th power mean $\P_p$. It is also obvious that Gini means are continuous in their entries and weights and $\G_{p,q}=\G_{q,p}$. According to Losonczi \cite{Los71a,Los71c}, Gini mean $\G_{p,q}$ is monotone for all $(p,q)$ with $pq\le 0$. Moreover it is known \cite[Corollary~4.2]{PalPas16} that 
\Eq{E:HGpq}{
\Hc(\G_{p,q})
=\begin{cases} 
\left( \dfrac{1-q}{1-p} \right)^{\frac1{p-q}} & p \ne q \text{ and } \min(p,q)\le0\le\max(p,q)<1, \\
e & p=q=0,\\
+\infty & \min(p,q)> 0\text{ or }\max(p,q)\ge1.
\end{cases}
}
In the case $\max(p,q)<0$ the value of $\Hc(\G_{p,q})$ is unknown (but finite).
Then by Theorem~\ref{thm:DevExt} for all $p,q\in\R$ Gini means admit unique extensions to integral means $\G_{p,q}^0$ is given by 

\Eq{*}{
  \G_{p,q}^0(f)
   :=\left\{\begin{array}{ll}
    \left(\dfrac{\int f(t)^p\:dt}
           {\int f(t)^q\:dt}\right)^{\frac{1}{p-q}} 
      &\mbox{if }p\neq q, \\[4mm]
     \exp\left(\dfrac{\int f(t)^p\ln f(t)\:dt}
           {\int f(t)^p\:dt}\right) \quad
      &\mbox{if }p=q,
    \end{array}\right.
}
where $f \colon J \to \R_+$ is a measurable function such that this means are well defined, i.e. 
\Eq{*}{
\G_{p,q}^0 &\colon \calL(\R_+)\cap L^p\cap L^q \to\R_+ \qquad &\text{ for }p \ne q,\\
\G_{p,q}^0 &\colon \calL(\R_+)\cap L^p\cap (L^p\log^+L) \to \R_+ \qquad &\text{ for }p = q.
}
By Theorem~\ref{thm:Hardy}, we obtain $\iHc(\G_{p,q}^0)=\Hc(\G_{p,q})$ whenever $pq\le 0$.
Applying test function $f(t)=e^{-t}$ we can check that this equality is also valid for $\min(p,q)>0$. Therefore \eq{E:HGpq} remains valid with $\Hc(\G_{p,q})$ replaced by $\iHc(\G_{p,q}^0)$. 
If $\max(p,q)\le 0$ then, as $\G_{p,q}^0$ is nondecreasing in $p$ and $q$ we get $\iHc(\G_{p,q}^0) \le \iHc(\G_{0,0}^0)=e$. In particular these integral means are Hardy. This is a partial solution of the problem posted in \cite[Problem~4, p.~89]{KufMalPer07}.

At the very end, let us mention the following theorem related to deviation means
\begin{thm}[\cite{PalPas18a}, Theorem~3.4]
Let $f\colon\R_+\to\R$ be a strictly increasing concave function with $f(1)=0$. Then $E(x,y):=f(\tfrac{x}{y})$ is a deviation, corresponding mean $\D_E$ is a Hardy 
mean if and only if 
\Eq{*}{
  \int_0^1f\Big(\frac{1}{t}\Big)dt<+\infty
}
and, if the above inequality holds, then its Hardy constant is the unique positive solution $c$ of the equation
\Eq{*}{
  \int_0^cf\Big(\frac{1}{t}\Big)dt=0.
}
\end{thm}
Applying Theorems~\ref{thm:DevExt} and \ref{thm:Hardy} above, we obtain that this theorem remains valid for the integral deviation means $\D_E^0$, too.


\end{document}